\documentclass[twoside,11pt,reqno]{amsart}
\usepackage{amsfonts,amsmath,amssymb,amscd,mathrsfs}
\usepackage{amsbsy,amsthm,latexsym,verbatim,enumerate,color}
\usepackage{bbding}
\usepackage{graphicx,floatrow}
\usepackage{CJK}
\usepackage{epsfig}
\usepackage[square,comma,sort,compress,numbers]{natbib}
\newtheorem{lemma}{Lemma}[section]
\newtheorem{theorem}[lemma]{Theorem}
\newtheorem{corollary}[lemma]{Corollary}

\newtheorem{remark}{Remark}

\theoremstyle{definition}
\definecolor{re}{rgb}{1,0.2,0.2}           
\definecolor{gr}{rgb}{0,1,0}
\definecolor{bl}{rgb}{0,0,0.6}
\definecolor{bl2}{rgb}{0,1,0}

\usepackage{youngtab}
\date{{\scriptsize Version 3 September 2018, printed \today}}
\begin{document}

\title[Graphs determined by their generalized spectrum]{New families of graphs determined by their generalized spectrum
}
\author{Fenjin Liu, Johannes Siemons, Wei Wang}

\address{F. Liu: School of Science, Chang'an University, Xi'an, P.R. China, 710046\newline
School of Mathematics and Statistics, Xi'an Jiaotong University, Xi'an, P.R. China, 710049\newline
School of Mathematics, University of East Anglia,  Norwich, Norfolk, NR4 7TJ, UK.}
\email{fenjinliu@yahoo.com}

\address{J. Siemons: School of Mathematics, University of East Anglia, Norwich, NR4 7TJ, UK}
\email{j.siemons@uea.ac.uk}

\address{W. Wang: School of Mathematics and Statistics, Xi'an Jiaotong University, Xi'an, P.R. China, 710049}
\email{
    wang\_weiw@163.com}

\maketitle

{\sc Abstract:}\,   We construct infinite families of graphs that are determined by their generalized spectrum. This construction is based on new formulae for the determinant of the walk matrix of a graph. The graphs constructed here all satisfy a lower divisibility for the determinant of their walk matrix.

\section{ Introduction}

Let $G$ be a graph on the vertex set $V$ with adjacency matrix $A$ and let $\overline{G}$ be its complement, with adjacency matrix $\overline{A}.$ Then the {\it spectrum} ${\rm spec}(G)$ of $G$ is the set of  all eigenvalues of $A,$ with corresponding multiplicities. It has been a longstanding problem to characterize graphs that are determined by their spectrum, that is to say: If  ${\rm spec}(H)={\rm spec}(G),$ does it follow that $H$ is isomorphic to $G?$ If this is the case then $G$ is {\it determined by its spectrum,} or a DS graph for short. In~\cite{Dam-Wga,Haemers-AlmostDS,Godsil-AlmostDS} it is conjectured that almost all graphs are DS, more recent surveys can be found in \cite{Dam-Dos}.

A variant of this problem concerns the {\it generalized spectrum} of $G$ which is given by the pair  $\big({\rm spec}(G),\,{\rm spec}(\overline G)\big).$ In this situation $G$ is said to be {\it determined by its generalized spectrum,} or a DGS graph for short, if   $\big({\rm spec}(H),\,{\rm spec}(\overline H)\big)=\big({\rm spec}(G),\,{\rm spec}(\overline G)\big)$ implies that $H$ is isomorphic to $G.$ For the most recent results on DGS graphs see \cite{LHMao-NewMethodDS} and \cite{Wang-ASimpleCriterion} in particular. To date only  a few families of graphs are known to be DS or DGS. This includes some almost complete graphs \cite{Camara-Almost}, pineapple graphs \cite{Hatice-Pineapple} and kite graphs \cite{Hatice-Kite}.  These are all known to be DS, and in particular DGS.  In addition all rose graphs are determined by their Laplacian spectrum except for two specific examples, see~\cite{He-Rose}.

In this paper we construct infinite sequences of DGS graphs from certain small starter graphs. This construction involves the {\it walk matrix} of a graph and recent results in~\cite{Wang-ASimpleCriterion}. Let $n$ be the number of vertices of $G$ and let $e$ be the column vector of height $n$ with all entries equal to $1.$ Then the {\it walk matrix} of $G$ is the $n\times n$ matrix
\begin{equation*}
W=\big[e,\,Ae,\,A^2e,\cdots, A^{n-1}e \big]
\end{equation*}
formed by  the column vectors $A^{i}e.$ In this paper we investigate the determinant of the  walk matrix of  $\overline{G}$ and of graphs obtained by joins and unions of $G$ with a single new vertex.

We prove that the determinant of the walk matrix of $G$ and that of $\overline{G}$ are the same up to a sign, see Theorem \ref{thm-detbarW}. We consider a graph $G$ and a new vertex $w$. In this situation the union $G\cup w$ and the join $G\vee w$ graphs can be defined, see Section 2. In Theorem \ref{thm-detWGcupv}
we obtain a formulae for the determinant of their walk matrices in terms of the determinants of $G,$ $\overline{G}$ and $W.$

In Wang \cite{Wang-GeneralizedRevisited} it is shown that $2^{\lfloor\frac{n}{2}\rfloor}$ divides $\det(W)$. Furthermore, if
\begin{equation*}
  \mathcal{C}:\quad \frac{\det(W)}{2^{\lfloor\frac{n}{2}\rfloor}}\text{\,\, is odd and square-free }
\end{equation*}
then $G$ is determined by its generalized spectrum \cite{Wang-ASimpleCriterion}. This shows that the divisor $2^{\lfloor\frac{n}{2}\rfloor}$  plays a special role for the determinant of the walk matrix, in some sense\,
$\mathcal{C}$\, is an extremal divisibility condition, see also Section 6.2 in \cite{LHMao-NewMethodDS}.

In Section~3 we give the construction of infinite families of graphs in which $\mathcal{C}$ holds, this is Theorems~\ref{thm-detWni} and~\ref{thm-detWn2to-n-above2}. This  construction provides in particular  infinite families of graphs which are determined by their generalized spectrum. Unlike constructing graphs with the same generalized spectrum (e.g. GM-Switching), very few methods are known for building DGS graphs. The results in this paper therefore contribute to our understanding 
of DGS graphs, they also give a partial answer a problem posed in \cite{LHMao-NewMethodDS}.

\medskip
All graphs in this paper are finite, simple  and undirected. Our notation follows the standard texts, for instance~\cite{Cvetkovic-IntroductionGraphSpec,Godsil-AlgebraicGraphTheory}. The vertex set of the graph $G$ is denoted $V.$
The  \emph{characteristic polynomial} $P(x)$ of $G$ is the characteristic polynomial of $A,$ thus $P(x)=\det(x I-A),$ and the \emph{eigenvalues} of  $G$ or $A$ are the roots of $P(x).$ 
When it is necessary to refer to a particular graph $H$ we denote its vertices, adjacency or walk matrix by $V(H),$ $A(H)$ or $W(H),$ etc.

\section{The determinant of the walk matrix}
We begin by giving a formula for the determinant of the walk matrix of a graph and its complement.

\medskip
\begin{theorem}\label{thm-detbarW}
Let $G$ be a graph on $n$ vertices with walk matrix $W.$ Then the walk matrix 
of its complement satisfies
\begin{equation*}
  \det(\overline{W})=(-1)^{\frac{n(n-1)}{2}}\det(W)\,\,.\end{equation*}
\end{theorem}

\begin{proof}
Let $A$ and $\overline{A}$ be the adjacency matrix of $G$ and $\overline{G}$ respectively. We show that for each $k$ the $k$-th column of $\overline{W}$ can be expressed as a linear combination of the first $k$ columns of $W$. This is true for the first and second columns of $\overline{W}$ since
\begin{equation*}
\overline{A}e=(J-I-A)e=(n-1)e-Ae,
\end{equation*}
where $J$ is the all-one matrix and $I$ is the identity matrix.
So we assume that the claim holds for some $k,$ that is, there exist numbers  $c_0,c_1,\ldots,c_{k-1}\in \mathbb{R}$ such that
\begin{equation*}
\overline{A}^{k-1}e=\sum_{i=0}^{k-1}c_iA^ie.
\end{equation*}
Since $JA^ie=(e^TA^ie)e$, we have
\begin{equation}\label{eq-detWCACk}
\begin{split}
\overline{A}^{k}e&=\overline{A}(\overline{A}^{k-1}e)\\
&=(J-I-A)(\sum_{i=0}^{k-1}c_iA^ie)\\
&=[\sum_{i=0}^{k-1}c_i(e^TA^ie)-c_0]e-\sum_{i=1}^{k-1}(c_{i-1}+c_i)A^ie-c_{k-1}A^ke.
\end{split}
\end{equation}
Therefore the above claim is true. In particular, for each $k=1,2,\ldots$, we conclude that the coefficient of the vector $A^ke$ is $-c_{k-1}=(-1)^k$. Substituting \eqref{eq-detWCACk} into $\overline{W}$ gives
\begin{equation*}
  \det(\overline{W})=\det \big[e,\,-Ae,\,A^2e, -A^{3}e,\cdots,\, (-1)^{n-1}A^{n-1}e \big]
\end{equation*}
and so the result follows from the multilinearity of the determinant.
\end{proof}

\medskip
By the same reasoning we have  the following  result for the leading principal submatrices of the walk matrix.

\medskip
\begin{corollary}
Let $W_1,W_2,\ldots,W_k$ (resp. $\overline{W}_1,\overline{W}_2,\ldots,\overline{W}_k$) be the first $k$ leading principal submatrices of the walk matrix $W$ (resp. $\overline{W})$ for $k=1,2,\ldots,n$. Then
\begin{equation*}
  \det(\overline{W}_k)=(-1)^{\frac{k(k-1)}{2}}\det(W_k)\,\,. \end{equation*}
\end{corollary}

\medskip
The graph $G$ is {\it controllable} if its walk matrix $W(G)$ is invertible. This property can also be characterized by the main eigenvalues and main eigenvectors  of the graph, see~\cite{Godsil-Controllable}. The relevance of controllability becomes clear from the recent work of O'Rourke and B. Touri~\cite{ORourke-AlmostAllControllable} who proved Godsil's conjecture~\cite{Godsil-Controllable} that asymptotically all graphs are controllable.
The  theorem above also implies the following well-known fact concerning controllable graphs:

\medskip
\begin{corollary}\cite{Godsil-Controllable}
A graph is controllable if and only if its complement is controllable.
\end{corollary}

\medskip Let $G$ be a graph and let $w$ be a new vertex, $w\not\in V(G).$ Then the {\it union} of $G$ and the singleton graph $\{w\}$, denoted by $G\cup w$, is the graph obtained from $G$ by adding $w$ as an isolated vertex. The {\it join} of $G$ and $\{w\}$, denoted by $G\vee w$, is the graph obtained from $G$ by adding the vertex $w$ and making it adjacent to all vertices of $G$. For these graph operations we have the following result.

\medskip
\begin{theorem}\label{thm-detWGcupv}
Let $G$ be a graph with adjacency matrix $A$ and walk matrix $W$. Then we have\\[-25pt]
\begin{enumerate}
\item[$(i)$] $\det(W(G\cup w))=\pm\det(A)\det(W)$ and
\item [$(ii)$] $\det(W(G\vee w))=\pm\det(\overline{A})\det(W)$.
\end{enumerate}
\end{theorem}

The  sign only depends on the position of the new vertex. For instance, if $\{w,v_{1},...,v_{n}\}$ are the vertices of $G\cup w$ then $\det(W(G\cup w))=\det(A)\det(W)$ and
$\det(W(G\vee w))=\det(\overline{A})\det(W)$.

\begin{proof}
Denote by $W'=\big[Ae,\,A^2e,\,\cdots, A^{n-1}e,\,A^ne \big]$. Then, taking the new vertex $w$ as the first vertex in $G\cup \{w\},$ we have
\begin{equation*}
  W(G\cup w)=\begin{bmatrix}
    1&0_{n\times 1}\\
    1_{n\times 1}&W'
  \end{bmatrix}.
\end{equation*}
Expanding the determinant of $W(G\cup w)$ along the first row gives
\begin{equation*}
\begin{split}
\det(W(G\cup w))&=\det(W')=\det(AW)=\det(A)\det(W).
\end{split}
\end{equation*}
If the new vertex $w$ is labelled even, then $\det(W(G\cup w))=-\det(A)\det(W)$.
This proves the first part of the theorem.

For the second part, first note that $G\vee w= \overline{\overline{G}\cup w}$. Now use  Theorem~\ref{thm-detbarW}\, to obtain
\begin{equation*}
\begin{split}
\det(W(G\vee w))&=\det (W(\overline{\overline{G}\cup w}))\\
&=(-1)^{\frac{n(n+1)}{2}}\det (W(\overline{G}\cup w))\\
&=\pm\det(\overline{A})\det(\overline{W})\\
&=\pm\det(\overline{A})\det(W).
\end{split}
\end{equation*}
This completes the proof.
\end{proof}

\medskip
\begin{remark}\label{rem-controllable}
 {\rm  A graph is {\it singular} if its adjacency matrix is singular, see \cite{Sciriha-Singular, S-Z}. The theorem has an interesting consequence for singular graphs: If $G$ is controllable then $G\cup w$ is controllable if and only if $G$ is not singular. Similarly, if $G$ is controllable then $G\vee w$ is controllable if and only if $\overline{G}$ is not singular.}
\end{remark}

\section{Constructing DGS graphs}

In this section we construct families of DGS graphs by using the union and join operations.
We begin with the Coefficient Theorem of Sachs. It relates the coefficient of  the characteristic polynomial of the graph to its  structure.  An \emph{elementary graph} is a graph in which each component is $K_2$ or a cycle.

\medskip
\begin{theorem}[Sachs Coefficients Theorem, see e.g.~\cite{Cvetkovic-IntroductionGraphSpec}]\label{thm-Sach}
  Let $G$ be a graph on $n$ vertices with characteristic polynomial $P(x)=x^n+c_1x^{n-1}+\cdots+c_{n-1}x+c_n$. Denote by  $\mathcal{H}_i$ the set of all elementary subgraphs of $G$ with $i$ vertices. For $H$ in $\mathcal{H}_i$ let $p(H)$ denote the number of components of $H$ and $c(H)$ the number of cycles in $H$. Then
\begin{equation*}
  c_i=\sum_{H\in\mathcal{H}_i}(-1)^{p(H)}2^{c(H)},\mbox{\qquad for all $i=1,\ldots,n$.}
\end{equation*}
\end{theorem}

\medskip
The following theorem due to Wang~\cite{Wang-ASimpleCriterion} characterizes certain DGS graphs by an arithmetic property of the determinant of their walk matrix.

\medskip

\begin{theorem}[Wang~\cite{Wang-ASimpleCriterion}]\label{thm-Wang-ASimpleCriterion}
Let $G$ be  a graph on $n$ vertices with $n\geq 6$ and walk matrix $W.$  Then $2^{\lfloor{\frac{n}{2}}\rfloor}$ divides $\det(W)$. Furthermore, if $2^{-\lfloor{\frac{n}{2}}\rfloor}\det(W)$ is odd and square-free then $G$ is determined  by is generalized spectrum.
\end{theorem}
\medskip

We are now able to state our next result.

\medskip
\begin{theorem}\label{thm-detWni}
Let $G_0,\,G_{1},\,G_{2},\,..$. be a sequence of graphs which satisfy the following conditions
\begin{equation*}
  G_{i}=\begin{cases}
    G_{i-1}\cup w_{i}\quad(\mbox{if $i\ge1$ is odd});\\
    G_{i-1}\vee w_{i}\quad(\mbox{if $i\ge1$ is even}).
  \end{cases}
\end{equation*}
Denote $n_{0}:=|V(G_{0})|$, $a:=|\det(A(G_{0}))|,$ \,$b:=|\det(W(G_{0}))|$ and  $p:=|\det(A(\overline{G}_1))|$. Then\begin{equation}\label{eq-dewWGi}|\det(W(G_{i}))|=a^{\lceil\frac{i}{2}\rceil}bp^{\lfloor\frac{i}{2}\rfloor} \quad \text{for all $i\geq 1$.}
\end{equation}
\end{theorem}
\begin{proof}
By Theorem \ref{thm-detWGcupv} (i) and (ii) we have
\begin{equation*}
\begin{split}
|\det(W(G_1))|&=|\det(W(G_0\cup w_1))|=|\det(A(G_0))\det(W(G_0))|=ab,
\end{split}
\end{equation*}
\begin{equation*}
\begin{split}
|\det(W(G_2))|&=|\det(W(G_1\vee w_2))|=|\det(A(\overline{G}_1))\det(W(G_1))|=abp.
\end{split}
\end{equation*}
Thus the result holds for $i=1,2$. Since the determinant of the adjacency matrix and the constant term of its characteristic polynomial are the same up to sign we use Theorem \ref{thm-Sach} to link it to the elementary spanning subgraphs.  By induction we have   $|\det(A(G_{2i}))|=a$ and  $|\det(A(\overline{G}_{2i+1}))|=p$ for any integer $i\ge 1$. Since all elementary spanning subgraphs $H(G_{2i})$ in $\mathcal{H}_{\substack{n_{0}+2i}}(G_{2i})$ must have $K_2=w_{2i}w_{2i-1}$ as a component,  there is a bijection between the  elementary spanning subgraph set $\mathcal{H}_{\substack{n_{0}+2i}}(G_{2i})$ and $\mathcal{H}_{\substack{n_{0}+2(i +1)}}(G_{2(i+1)})$ i.e., there is a bijection
\begin{equation}\label{eq-constantbijection}
\begin{split}
&f:\mathcal{H}_{\substack{n_{0}+2i}}(G_{2i})\leftrightarrow \mathcal{H}_{\substack{n_{0}+2(i+1)}}(G_{2(i+1)}) \, \,\text{with} \\
&f(H(G_{2i}))=H(G_{2i})\cup w_{2i+2}w_{2i+1}.
\end{split}
\end{equation}
Note that the component $K_2=w_{2i+2}w_{2i+1}$ only changes the sign of the constant  coefficient of the characteristic polynomial. Therefore each pair of  elementary spanning subgraphs $H(G_{2i})$ and $H(G_{2i})\cup w_{2i+2}w_{2i+1}$ contribute opposite signs to the constant terms of  $P_{\substack{G_{2i}}}(x)$ and $P_{\substack{G_{2(i+1)}}}(x)$, respectively. Hence $|\det(A(G_{2(i+1)}))|=a$. Analogously we have   $|\det(A(\overline{G}_{2i+1}))|=p$. Equation  \eqref{eq-dewWGi} follows by applying Theorem \ref{thm-detWGcupv} repeatedly.
\end{proof}

From this result we obtain infinite sequences of graphs that are determined by their generalized spectrum:

\medskip
\begin{theorem}\label{thm-detWn2to-n-above2}
Let $G_0,\,G_{1},\,G_{2},\,\ldots$ be as in Theorem \ref{thm-detWni} and denote $|V(G_{0})|$ by $n_{0}.$ Suppose that  $\{a,p\}=\{1,2\}$ and that $b\cdot2^{-\lfloor\frac{n_0}{2}\rfloor}$ is an odd square-free integer. Then  the graphs $G_{i}$ are determined by their  generalized spectrum, for all $i\geq 1$.
\end{theorem}
\begin{proof}
It is easy to see that $2^{-\lfloor\frac{n_{0}+i}{2}\rfloor}\cdot \det(W(G_i))$ is odd and square-free when $\{a,p\}=\{1,2\}$ and $b\cdot {2^{-\lfloor\frac{ n_0}{2}\rfloor}}$ is an odd square-free integer. Thus our result follows  from Theorem~\ref{thm-Wang-ASimpleCriterion} and Theorem~\ref{thm-detWni}.
\end{proof}

We mention several remarks and  open problems.

\smallskip
\begin{remark}\label{prop C} {\rm In the introduction we discussed the importance of the property ${\mathcal C}$ as an extremal divisibility condition for the walk matrix of a graph. All the graphs $G_{i}$ in Theorem~\ref{thm-detWn2to-n-above2}\, now satisfy the condition ${\mathcal C}.$}\end{remark}

\begin{remark}\label{rem-odd-square-free} {\rm There are indeed many graphs $G_{0}$ that are suitable starters for such sequences. Among the $112$ connected graphs on six vertices there are $8$ controllable graphs, labelled $59, 77$ in~\cite{Cvetkovic-A-Table-V6} with $|\det(W)|=3\cdot 2^3$ and $46, 60, 67, 85, 87, 98$ in~\cite{Cvetkovic-A-Table-V6} with $|\det(W)|=2^3.$ These graphs therefore have property $\mathcal{C}$ and are controllable.
According to Theorems~\ref{thm-Wang-ASimpleCriterion}, \ref{thm-detWni} and \ref{thm-detWn2to-n-above2} we obtain infinite series of graphs based on 6 of these $8$ graphs as an initial $G_0$, namely graphs labelled $59, 77,67, 85, 87, 98$. In particular, all graphs in these series have property $\mathcal{C}$ and hence are DGS graphs.} \end{remark}

\medskip
\begin{remark}\label{rem-2tonabove2}
{\rm Mao et al.~\cite{LHMao-NewMethodDS} stated the problem of characterizing graphs with $|\det(W)|= 2^{\lfloor \frac{n}{2}\rfloor}$. This condition can be viewed
as a strengthening of condition $\mathcal{C}$
and is worth investigation independently.
By Theorems~\ref{thm-detWni} and~\ref{thm-detWn2to-n-above2} we can construct infinite families of such graphs based the controllable graphs labelled $67, 85, 87, 98$ in ~\cite{Cvetkovic-A-Table-V6}.}
\end{remark}

\medskip
\begin{remark}{\rm
Theorems \ref{thm-detbarW},~\ref{thm-Wang-ASimpleCriterion} and~\ref{thm-detWn2to-n-above2} imply that the complement of any $G_i$ constructed in Theorem \ref{thm-detWni} is also DGS, for all $i\ge 0.$}
\end{remark}

\begin{figure}[h!]
  \centering
 \includegraphics[width=14cm]{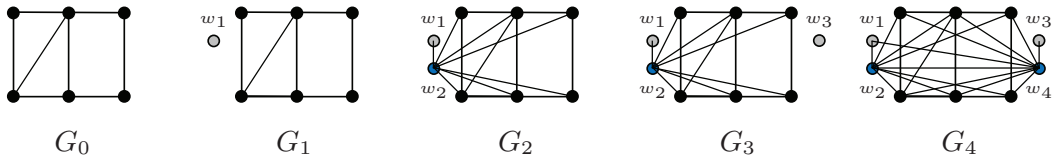}\\
  \caption{Constructing a family of DGS graphs from $G_0\cong\#67$}\label{Fig1}
\end{figure}

\medskip
\begin{remark}\label{ex-DGS}{\rm
Let $G_0$ be the graph labelled  $67$ in~\cite{Cvetkovic-A-Table-V6}, see Fig.\ref{Fig1}. We show schematically the first five DGS graphs obtained from Theorem~\ref{thm-detWni}. It is easy to verify that these graphs satisfy $|\det(W)|=\lfloor\frac{n}{2}\rfloor$.}
\end{remark}

\vspace{-5pt}

{\bf Open Problems.} We conclude  with two open problems:\\
(i)\,\,Find other constructions of graphs satisfying property $\mathcal{C}$.\\
(ii)\,Determine the generalized spectrum of graphs with property $\mathcal{C}$ and classify such graphs.


\subsection*{Acknowledgements}
The first author would like to express his gratitude to Professor Qiongxiang Huang for guiding him towards  Spectral Graph Theory and the faculty of the School of Mathematics at UEA for accepting  him into their friendly research environment.

This work is supported by the National Natural Science Foundation of China (Nos. 11401044, 11471005, 11501050), Postdoctoral Science Foundation of China (No. 2014M560754), Postdoctoral Science Foundation of Shaanxi, the Fundamental Research Funds for the Central Universities (No. 300102128201) and the Foundation of China Scholarship Council (No. 201706565015).

\end{document}